\documentclass[12pt, reqno]{amsart}
\usepackage[active]{srcltx}
\usepackage[mathscr]{eucal}
\usepackage{amsmath}
\usepackage{amssymb}
\usepackage{color}

\title[Space-time analyticity for inviscid fluid dynamic equations]{Space and time analyticity for inviscid equations of fluid dynamics}
\author{Animikh Biswas}
\address{Department of Mathematics \& Statistics, 
University of Maryland, Baltimore County,
1000 Hilltop Circle, Baltimore, Maryland - 21250, USA.}
\email{abiswas@umbc.edu}
\author{Joshua Hudson$^\dagger$}
\address{Department of Mathematics \& Statistics, 
University of Maryland, Baltimore County,
1000 Hilltop Circle, Baltimore, Maryland - 21250, USA.}
\email{JoshuaHudson@umbc.edu}
\thanks{$^\dagger$\, Corresponding author.}
\subjclass{Primary 35Q35; Secondary 35Q30; 76D09}
 \keywords{Euler Equations; Analyticity; Gevrey classes; Inviscid Equations.}

\newtheorem{theorem}{Theorem}[section]

\newtheorem{prop}[theorem]{Proposition}
\newtheorem{cor}[theorem]{Corollary}

\newtheorem{rem}{Remark}

\newcommand{\p}{\mathbb P}
\newcommand{\comments}[1]{}

\renewcommand{\phi}{\varphi}
\newcommand{\R}{\mathbb R}
\newcommand{\C}{\mathbb C}

\newcommand{\g}{\gamma}
\newcommand{\G}{\Gamma}
\newcommand{\nn}{\nonumber}

\newcommand{\ka}{\kappa }
\newcommand{\Z}{\mathbb Z}
\newcommand{\N}{\mathbb N}
\newcommand{\cal}{\mathcal }

\newcommand{\ra}{\rightarrow}

\newcommand{\hh}{\mathbb H}
\newcommand{\sob}{H^2(\Omega)}
\newcommand{\dt}{\frac{d}{dt}}
\newcommand{\ds}{\frac{d}{ds}}
\renewcommand{\l}{\langle}
\renewcommand{\r}{\rangle}
 
\newcommand{\be}{\begin{equation}}
\newcommand{\ee}{\end{equation}}
\newcommand{\bes}{\begin{equation*}}
\newcommand{\ees}{\end{equation*}}
\newcommand{\ba}{\begin{eqnarray}}
\newcommand{\ea}{\end{eqnarray}}
\newcommand{\bas}{\begin{eqnarray*}}
\newcommand{\eas}{\end{eqnarray*}}
\newcommand{\wh}{\widehat}
\newcommand{\cW}{\cal W}
\newcommand{\cL}{{\cal L}}
\newcommand{\cD}{{\cal D}}
\newcommand{\length}{l}
\newcommand{\Cwien}[1][s]{c_{#1}}
\newcommand{\Span}{\text{span}}

\newcommand{\bu}{\mathbf u}
\newcommand{\bb}{\mathbf b}
\newcommand{\bv}{\mathbf v}
\newcommand{\bw}{\mathbf w}
\newcommand{\wu}{\wh{\bu}}
\newcommand{\wv}{\wh{\bv}}
\newcommand{\ww}{\wh{\bw}}
\def\e{{\mathbf e}}
\def\h{{\mathbf h}}
\def\j{{\mathbf j}}
\def\k{{\mathbf k}}
\def\x{{\mathbf x}}
\def\0{{\mathbf 0}}
\def\z{{\mathbf z}}
\def\tF{\widetilde{F}}

\newcommand{\bn}{\boldsymbol{|\!\!|}}
\newcommand{\tZ}{\widetilde{\mathbb Z}}
\def\c1{(1+\g )^{\G_0 \alpha}}

\def\intav#1{\mathchoice
          {\mathop{\vrule width 6pt height 3 pt depth -2.5pt
                  \kern -9pt \intop}\nolimits_{\kern -6pt#1}}%
          {\mathop{\vrule width 5pt height 3 pt depth -2.6pt
                  \kern -6pt \intop}\nolimits_{#1}}%
          {\mathop{\vrule width 5pt height 3 pt depth -2.6pt
                  \kern -6pt \intop}\nolimits_{#1}}%
          {\mathop{\vrule width 5pt height 3 pt depth -2.6pt
                  \kern -6pt \intop}\nolimits_{#1}}}

\newcommand{\charfn}[1]{{\raisebox{1.2pt}{\mbox{$\chi
_{\kern-1pt\lower3pt\hbox{{$\scriptstyle{#1}$}}}$}}}}

\begin{document}
\begin{abstract}
We show that  solutions to a large class of inviscid equations, in Eulerian variables, 
extend as  holomorphic functions of time, with values in a Gevrey class (thus space-analytic), and are solutions of complexified versions of the said equations. The class of equations we consider includes those of fluid dynamics such as the Euler, surface quasi-geostrophic, Boussinesq and magnetohydrodynamic equations, as well as other equations with  analytic nonlinearities. The initial data are assumed to belong to a {\it Gevrey class}, i.e., analytic in the space variable.
Our technique follows that of the seminal work of  Foias and Temam (1989), where they introduced the so-called Gevrey class technique for the Navier-Stokes equations to show that the solutions of the Navier-Stokes equations extend as holomorphic functions of time, in a complex neighborhood of $(0,T)$, with values in a Gevrey class of functions (in the space variable). We show a similar result for a wide class of inviscid models, while obtaining an {\it explicit estimate of the domain of analyticity}.
\end{abstract}
\maketitle
\section{Introduction}
It is well-known that solutions to a large class of dissipative  equations 
are analytic in space and time \cite{BaeBi, BiSw, Bradshaw, feti, ft1, ft, gk1, ma, ot}.
In fluid dynamics,  space analyticity radius has a physical interpretation. It denotes a length scale below which the viscous effects dominate and the Fourier spectrum decays exponentially, while above it, the inertial effects dominate \cite{doti}. 
This fact concerning exponential decay can be used to show that the Galerkin approximation converges exponentially to the exact solution \cite{dot}.
Other applications of analyticity radius occur in establishing sharp temporal decay rates of solutions in higher Sobolev norms \cite{Bi, ot}, establishing geometric regularity criteria for solutions, and in measuring the spatial complexity of fluid flows \cite{gr, ku}. Likewise, time analyticity also has several  important applications including establishing backward uniqueness of trajectories \cite{cf}, parameterizing turbulent flows by finitely many space-time points \cite{kr} and numerical determination of the attractor \cite{fjl}.

 Space and time analyticity of inviscid equations, particularly the Euler equations, has received considerable attention recently, as well as in the past. Space analyticity for Euler, in the Eulerian variables,  was considered for instance in \cite{bb, bbz, kv1, kv2, lo}, while in \cite{am, delort, lebail} real analyticity in the time (and space) variable is established using harmonic analysis tools. In the above mentioned works, the initial data are taken to be analytic in the space variable. By contrast, in a recent work \cite{cvw}, it is shown that the Lagrangian trajectories are  real analytic (in time), even though the initial velocity fields are slightly more regular than Lipschitz in the  space variable. Similar results also appear elsewhere; see for instance in \cite{serfati, shnirel,zf} and the references therein. Additionally, the contrast between the analytic properties in the Eulerian and Lagrangian variables has been considered recently in \cite{cvk}.
 
In this paper, we show that solutions of the Euler, as well as the inviscid versions of the SQG, Boussinesq, MHD, and similar equations with analytic nonlinearities, with analytic initial data, extend as solutions of the complexified versions of the equations, as holomorphic functions of time, with values in a suitable Gevrey class of functions in the space variable. Since belonging to a Gevrey class is equivalent to a function being (complex) analytic, this immediately establishes that the solutions  extend as  holomorphic functions of both space and time. In contrast to, for instance, the results in \cite{am, delort}, we not only obtain holomorphic extensions (as opposed to real in time analyticity in \cite{am, delort}) but also obtain {\it explicit estimates on the domain of (time) analyticity,} while in \cite{am, delort}, the region is given implicitly in terms of the {\it flow map} generated by the solutions. Our approach follows \cite{ft}, in which the desired results are obtained for the Navier-Stokes equations. We also make use of the ideas introduced in \cite{lo} and \cite{kv1}.

It should be noted that unlike their ``real" counterparts, the complexified inviscid models are not known to conserve ``energy", which is due to the fact that the complexified nonlinear terms do not in general possess cancellation properties akin to their real counterparts. Yet, as in \cite{lo}, the mild dissipation generated due to working in a Gevrey class setting is enough for local existence for the complexified versions of these inviscid models.

The paper is organized as follows. In section 2, we introduce the notation and discuss requisite results. In sections 5-7, we respectively consider the Euler equations, the inviscid surface quasi-geostrophic equations, the inviscid Boussinesq equations, the inviscid magnetohydrdynamic equations and an equation with an anlytic nonlinearity.

\section{Preliminaries}  \label{prelims}
\subsection{Notation and Setting for Incompressible Hydrodynamics.}
In this article, we will consider several evolutionary (incompressible) fluid dynamic models including the incompressible Euler equations, the surface quasi-geostrophic equation (SQG), the Boussinesq equations and the magnetohydrodynamic equations (MHD). In all the cases, these equations will be considered on a spatial domain  $\Omega=[0,\length]^d, d\in \N$, 
 and supplemented with the space periodic boundary condition (with spatial period $\length$), i.e., the phase space will comprise of scalar-valued or vector-valued functions, which are periodic in the space variable $\x$ with period $\length$ in all spatial directions.
 For notational simplicity, we will assume 
 \[
\length=2\pi,\ \mbox{and therefore,}\  
\ka_0:=\frac{2\pi}{\length}=1. 
\]
The inner product on $L^2(\Omega):=\left\{u:\Omega\to\R ,\  \int_\Omega |u(\x)|^2 d\x < \infty \right\}$
is denoted $\l \cdot\, , \cdot \r$ and the corresponding $L^2-$norm
will be  denoted as $\|\cdot \|$. As usual, the Euclidean length of a vector in $\R^d (\, \mbox{or}\, \C^d)$  is denoted by $|\cdot|$.

For a  function $\bu:\Omega \ra \R^d (\, \mbox{or}\, \C^d)$, 
its Fourier coefficients are defined by
\[
\widehat{\bu}(\k)= \frac{1}{(2\pi)^d} \int_{\Omega} \bu(\x) e^{-\imath \kappa_0 \k \cdot \x} d\x\  (\k \in \Z^d).\
\]
Then by the Parseval identity,
\[
\|\bu\|^2 = (2\pi)^{d} \sum_{\k\in \Z^d} |\wu(\k)|^2.
\]

In all the models we consider, if the space average of the initial data is zero, then the space average remains zero for all future times, under the evolution. Therefore, we will always make the additional assumption that the elements of the phase space have space average zero (over the spatial domain $\Omega$). In terms of the Fourier coefficients, this amounts to the condition $\wu({\mathbf 0})={\mathbf 0}$ (which is then preserved under the evolution). 

We will denote
\[
\dot{L}^2(\Omega)=\left\{u \in L^2(\Omega): \int_\Omega u(\x)\, d\x=0, \ \mbox{or equivalently,}\ 
\wh{u(\k)}=0\right\}.
\]
In the case of incompressible fluid dynamics, the phase space $H$ is given by 
\[
H=\left\{\bu \in (\dot{L}^2(\Omega))^d,\ \nabla \cdot \bu = \0\right\},
\]
where the derivative is understood in the distributional sense. Using Fourier coefficients, the space $H$ can alternatively be characterized by
\[
H=\{\bu \in (L^2(\Omega))^d,\ \wu(\0)=\0,\ \k \cdot \widehat{\bu}(\k)=\0,\ \wu(-\k)= \overline{\widehat{\bu}(\k)}\}.
\]

Note that the space $(-\Delta)(H \cap \hh^2) \subset H$. The Stoke's operator, $A$, with domain ${\cal D}(A)=H \cap \hh^2$, is defined to be $A=(-\Delta)|_{\cal D(A)}$. The Stoke's operator $A$ is positive and self adjoint with a compact inverse. It therefore admits a unique,  positive square root, denoted $A^{1/2}$, with domain $V$, where
 the space $V$ is characterized by
\[
V=\{\bu \in H: \|A^{1/2}\bu\|^2 = (2\pi)^d\sum_{\k \in \tZ^d}|\k|^2|\wu(\k)|^2 < \infty  \},
\]
where $\tZ^d=\Z^d \setminus \{\0\}$. The spectrum of $A$ comprises of 
eigenvalues $0< 1=\lambda_1 \le \lambda_2 \le \cdots $, where, for each $i$,  $\lambda_i \in \{|\k|^2: \k \in \tZ^d\}$. The set of eigenvectors $\{\e_i\}_{i=1}^\infty$, where $\e_i$ is an eigenvector corresponding to the eigenvalue $\lambda_i$, form an orthonormal basis of $H$. We will denote $H_N =\Span\{\e_1, \cdots ,\e_n\}$.

It is easy to see that the dual $V'$ of $V$ is given by 
\[
V'=\{ \bv \in {\mathscr D}: \wv (\k)=\overline{\wv(-\k)},\ \wv(\0)=\0,\ \sum_{\k \in \tZ^d} 
\frac{|\wv(\k)|^2}{|\k|^2} < \infty \},
\]
where ${\mathscr D}$ denotes the space of distributions. The duality bracket between $V$ and $V'$ is given by
\[
{}_{V}\l \bu , \bv \r_{V'}= \sum_{\k \in \tZ^d} \wu(\k)\overline{\wv(\k)},\quad \bu \in V, \bv \in V'.
\]

It should be noted that 
\[
\|A^{s/2}\bu\|^2 =(2\pi)^d\sum_{\k \in \tZ^d}|\k|^{2s}|\wu(\k)|^2\ \mbox{ for}\ 
\bu \in {\cal D}(A^{s/2})=(\dot{L}^2(\Omega))^d \cap \hh^s.
\]
\comments{where $A=-\Delta $ is the Laplacian, acting on  functions in $\dot{L}^2(\Omega) \cap \h^2(\Omega)$, where $\h^s(\Omega)$ denotes the usual ($L^2-$based) Sobolev spaces.} 
It is well known that $\|A^{s/2}\cdot\|\sim \|\cdot \|_{\hh^s}$
on ${\cal D}(A^{s/2})$, and the Poincar\'{e} inequality holds, i.e., 
\be \label{poincare}
\|A^{s/2}\bu\| \sim \|\bu\|_{\hh^s}\ \mbox{and}\  \|A^{s/2}\bu\| \ge \ka_0^s \|\bu\| \quad 
(\bu \in (\dot{L}^2(\Omega))^d \cap (\hh^s(\Omega))^d).
\ee
Using the Sobolev and interpolation inequalities, we also have
\be \label{Sob}
\|\bu\|_{L^p} \lesssim \|A^{s/2}\bu\| \le \|\bu\|^{1-s}\|A^{1/2}\bu\|^s,\ p = \frac{2d}{d-2s},\ 
1 \le d \le 4,
\ee
where $a\lesssim b$ has the same meaning as when we write $a\le Cb$, where $C$ is used to denote a general constant, which may only depend on $\length$ or $d$ and no other parameter in the problem. 

We will find it useful to define the so-called Wiener algebra, 
\be  \label{W}
{\cal W}:=\{\bu \in H: \|\bu\|_{\cW}:=\sum_\k |\wu(\k)| < \infty \}.
\ee
Clearly, from the expression of $\bu$ in terms of its Fourier series,
 $\bu (\x)= \sum_{\k \in \tZ^d} \wu(\k)e^{\imath \k \cdot \x}$, it immediately follows that
 \[
 \|\bu\|_{L^\infty} \le \|\bu\|_{\cW}.
 \]
 In addition, we have the elementary inequality
 \be  \label{Sob1}
 \|\bu\|_{L^\infty} \le \|\bu\|_{\cW} \le 
 \tfrac{2\pi^{d-1}}{\length^{d}}\tfrac{2s-d+1}{2s-d}
 \|A^{s/2}\bu\|, \ s > \frac d2.
 \ee
 We will be using \eqref{Sob1} with $s=r-\frac12$ for a number $r>\frac{d}{2}$, and so for readability we will define $\Cwien[r] = \tfrac{2\pi^{d-1}}{\length^{d}}\tfrac{2(r-\frac12)-d+1}{2(r-\frac12)-d}=\tfrac{1}{\pi2^{d-1}}\tfrac{2r-d}{2r-1-d}$.

 \comments{
 $\Gamma$ is the gamma function, and in the cases of d=2,3, we have:
 \be
 \tfrac{\pi^{d/2}d}{\Gamma(1+d/2)}=
 \begin{cases}
	 2\pi, \quad \text{ if } d=2\\
	 4\pi, \quad \text{ if } d=3
 \end{cases}
 \ee
 }

 Often, in equations such as the SQG or the Boussinesq, we will be interested in the evolution of a scalar variable $\eta$. In these cases, the phase space will be 
 $\dot{L}^2(\Omega)$ and $A=(-\Delta) $ acting on ${\cal D}(A)=\dot{L}^2(\Omega) \cap \sob$. We will not make any notational distinction between the operator $A$ in these cases and the Stoke's operator, which will be understood from context. The inequalities \eqref{poincare}, \eqref{Sob} and \eqref{Sob1} are still valid here.
 
 \subsection{Gevrey Classes}\label{Gevrey_Classes}
 Let $0< \beta <\infty $ and fix $r>0$. We denote the Gevrey norm by
\bes
\|\mathbf f\|_{\beta}=\|A^{r/2}e^{\beta A^{1/2}} \mathbf f\|.
\ees
The Gevrey norm is characterized by the decay rate of higher order derivatives, namely, 
if $\|{\mathbf f}\|_{\beta} < \infty $ for some $\beta>0$, then we have the higher derivative estimates 
\begin{gather}  \label{gevreydecay}
\|{\mathbf f}\|_{\hh^{r+n}} \le \left(\frac{n!}{\beta^n}\right)\|\mathbf f\|_{\beta}\ 
\mbox{where}\  n \, \in \, {\mathbb N} .
\end{gather}
In particular,   $\mathbf f$ in (\ref{gevreydecay}) is  analytic with (uniform) analyticity radius 
$\beta$.
For  the above mentioned facts including  (\ref{gevreydecay}), see Theorem 4 in \cite{lo} and Theorem 5 in \cite{ot}.

 \subsection{Complexification.} In order to extend the solutions of the equations to complex times, we need to complexify the associated phase spaces and operators. Accordingly, let $\cL$ be an arbitrary, real, separable Hilbert space with (real) inner-product $\l\cdot\, , \cdot\r$. The complexified Hilbert space $\cL_\C$ and the associated inner-product is given by
 \[
 \cL_\C=\{\bu=\bu_1+\imath\bu_2: \bu_1,\bu_2 \in \cL\},
 \]
 and for $\bu, \bv \in \cL_\C$ with $\bu=\bu_1+\imath \bu_2, \bv=\bv_1+\imath \bv_2$, the complex inner-product is given by
 \[
 \l\bu,\bv\r_\C=\l\bu_1,\bv_1\r+\l\bu_2,\bv_2\r+\imath[\l\bu_2,\bv_1\r-\l\bu_1,\bv_2\r].
 \]
 Observe that the complex inner-product $\l\cdot \, , \cdot\r_\C$ is linear in the first argument while it is conjugate linear in the second argument.
 If $A$ is a linear operator on $\cL$ with domain ${\cal D}(A)$, we extend it to a linear operator $A_\C$ with domain $\cD(A_\C)=\cD(A)+\imath \cD(A)$ by
 \[
 A_\C(\bu_1+\imath \bu_2)=A\bu_1+ \imath A\bu_2, \bu_1,\bu_2 \in \cD(A).
 \]
 Henceforth, we will drop the subscript notation from the complexified operators and inner-products and denote $A_\C$ and $\l\cdot\, , \cdot\r_C$ respectively as $A$ and 
 $\l\cdot\, , \cdot\r$, but will retain the subscript in the notation of the corresponding complexified Hilbert spaces.
 
\section{Incompressible Euler Equations}  \label{sec:Euler}

The incompressible Euler equations, on a spatial domain
$\Omega=[0,2\pi]^d,\ d=2\, \text{or}\, 3$, are given by
\begin{subequations}  \label{euler}
\begin{alignat}{2}
\partial_t \bu + (\bu \cdot \nabla) \bu + \nabla p &=0, 
\quad&& \text{in }\Omega\times \R_+,\\
\nabla \cdot \bu &=0,
\quad&& \text{in }\Omega\times \R_+,\\
\bu(\x,0)&=\bu_0(\x), \quad&& \text{in }\Omega,
\end{alignat}
\end{subequations}
where $\bu=\bu(\x,t)$ denotes the fluid velocity at a location $\x \in \Omega$ and time 
$t\in\R_+:=[0,\infty)$ and $p=p(\x,t)$ is the fluid pressure. Since its introduction in \cite{euler},
it has been the subject of intense research both in analysis and mathematical physics; see \cite{BaTi, majB} for a survey of recent results on \ref{euler}. 
As discussed in section \ref{prelims}, we supplement \eqref{euler}
with the space periodic boundary condition with space period $2\pi$, i.e., the functions
$\bu$ and $p$ are periodic with period $2\pi$ in all spatial directions.

We will also denote
\be  \label{nonlinterm}
B(\bu,\bv)=\p \left(\bu\cdot \nabla \bv\right)= \p\, \nabla \cdot (\bu \otimes \bv),
\ee
where $\p:(\dot{L}^2(\Omega))^d \ra H$ is the Leray-Helmholtz orthogonal projection 
operator onto the closed subspace $H$ of $(\dot{L^2}(\Omega))^d$. From \eqref{Sob}, it readily follows that if $\bu, \bv \in V$, then $\|\bu \otimes \bv\| < \infty$ and consequently,
$B(\bu,\bv) \in V'$.

The functional form of the incompressible Euler equations  is given by 
\be \label{eulerfuncform}
\dt \bu +  B(\bu,\bu)=0.
\ee
We will consider the complexified Euler equation given by
\be \label{cxeuler}
 \frac{d\bu}{d \zeta} + B_\C(\bu,\bu)=0, \bu(0)=\bu_0, 
 \ee
 where, for $\bu=\bu_1+\imath \bu_2, \bv=\bv_1+\imath \bv_2 \in H_\C$, the complexified nonlinear term is given by
 \[
 B_\C(\bu,\bv):= B(\bu_1,\bv_1)-B(\bu_2,\bv_2) + \imath [B(\bu_1,\bv_2)+B(\bu_2,\bv_1)].
 \]
 As before, we will drop the subscript and write $B=B_\C$.

 In the following, let $r > \frac {d+1}{2}$ be fixed, and we will consider the corresponding Gevrey norm, $\|\cdot\|_\beta$, as defined in Section~\ref{Gevrey_Classes}.

\comments{
 \[
 \|\bu\|_{\beta}^2:=\|A^{r/2}e^{\beta A^{1/2}}\bu\|^2
 = \sum_{\k \in \tZ^d} |\k|^{2r}e^{2\beta |\k|}|\wu(\k)|^2.
 \]
 }

 \begin{theorem}  \label{thm:euler}
 Let $\beta_0>0$ be fixed, and let $\bu_0$ be such that $\|\bu_0\|_{\beta_0} < \infty $.  The complexified Euler equation \eqref{cxeuler} admits a unique solution in the region 
 \be  \label{cxregion}
 {\cal R}= \left\{\zeta=se^{i\theta}: 
 \theta \in [0,2\pi), 
 0 < s < \frac{C\beta_0}{2^r\Cwien[r]\|\bu_0\|_{\beta_0}}\right\}.
 \ee
 \end{theorem}
 For the Euler equations in the real setting, it is well known  that if the Beale-Kato-Majda condition \cite{bkm}, $\int_0^T \|\nabla \times \bu \|_{L^\infty} < \infty $, is satisfied, and there exists $\beta_0$ such that $\|A^{r/2}e^{\beta_0A^{1/2}}\bu_0\|<\infty$, then there exists a continuous function 
 $\beta(t)>0$ on $[0,T]$ such that \cite{bb, kv1, kv2, lo}
 \[
  \sup_{[0,T]}\|A^{r/2}e^{\beta(t) A^{1/2}}\bu(t)\| < \infty .
  \]
 In this case, $\bu$ extends as a holomorphic function satisfying \eqref{cxeuler} in a neighborhood of $(0,\infty)$ in 
 $\C$. More precisely we have the following:

 \begin{cor}
 If there exists a continuous function $\beta(\cdot)>0$ on $[0,T]$ such that 
 \be  \label{finite}
 M:= \sup\limits_{t \in [0,T]} \|\bu(t)\|_{\beta(t)} < \infty , 
 \ee
 where $\bu$ is a solution of \eqref{eulerfuncform}, then $\bu(\cdot)$ extends as a holomorphic function in a complex 
 neighborhood of $(0,T)$, satisfying \eqref{cxeuler}.
  \end{cor}
 \begin{proof}
	 Let $\beta_0 = \inf_{t \in [0,T]}\beta(t) > 0$. Then by Theorem~\ref{thm:euler}, $\bu$ extends as a holomorphic function in a complex neighborhood of $(0,\varepsilon)$, where $\varepsilon = \frac{C\beta_0}{2^r\Cwien[r]M}$. The proof follows by reapplying Theorem~\ref{thm:euler} with $\bu_0 = \bu(t_0)$, for each $t_0\in\{\frac{\varepsilon}{2},\frac{2\varepsilon}{2},\frac{3\varepsilon}{2},\dots\}\cap[0,T]$.

 \end{proof}

 \comments{
  \begin{proof}
  Let $\beta_0 = \inf_{t \in [0,T]}\beta(t) >0$. Let $t_0 \in (0,T)$. The proof immediately follows by applying Theorem \ref{thm:euler} for $\bu_0=\bu(t_0 - \frac{\beta_0}{4M})$.
  \end{proof}
 }
 
Before proceeding with the proof of the theorem, we will need
the following estimate of the nonlinear term.
\begin{prop}  \label{prop:nocancellation}
Let $\bu \in H_\C$ with $\|A^{1/4}\bu\|_{\beta} < \infty$. Then,
\be  \label{nocancellation}
|\l B(\bu,\bu),A^{r}e^{2\beta A^{1/2}}\bu\r| \lesssim 2^r\Cwien[r]\|\bu\|_\beta \|A^{1/4}\bu\|_\beta^2.
\ee
\end{prop}
\begin{proof}
Observe that for $\h+\j=\k,\ \h, \j, \k \in \tZ^d$, we have
\[ 
	|\k|^r \le 2^{r-1}(|\h|^r + |\j|^r).
\]
Thus,
\ba
\lefteqn{ |\l B(\bu,\bv),A^{r}e^{2\beta A^{1/2}}\bw\r| } \nn \\
& &\le  \sum_{\h+\j-\k=\0}|\wu(\h)||\j||\wv(\j)||\k|^{2r}|\ww(\k)|e^{2\beta |\k|}
\nn \\
& & \le \: 2^{r-1}\hspace{-10pt}\sum_{\h+\j-\k=\0}|\h|^r|\wu(\h)||\j||\wv(\j)||\k|^{r}|\ww(\k)|e^{2\beta |\k|} \nn \\
& & \qquad + \quad 2^{r-1}\hspace{-10pt}\sum_{\h+\j-\k=\0}|\wu(\h)||\j||\j|^r|\wv(\j)||\k|^{r}|\ww(\k)|e^{2\beta |\k|}.  \label{intermed1}
\ea
Because $\j, \h, \k \neq \0$, we have $\min\{|\j|, |\h|, |\k|\} \ge 1$ and therefore,
\be  \label{ineq1}
|\j| \le |\h|+|\k| \le 2 |\h| |\k| \ \mbox{which implies}\ |\j|^{\frac{1}{2}} \lesssim |\h|^{\frac{1}{2}}|\k|^{\frac{1}{2}}.
\ee
From \eqref{intermed1} and \eqref{ineq1}, we have
\ba
\lefteqn{ |\l B(\bu,\bv),A^{r}e^{2\beta A^{1/2}}\bw\r| } \nn \\
& & \le \: 2^{r-1/2}\hspace{-10pt} \sum_{\h+\j-\k=\0}e^{\beta|\h|}|\h|^{r+\frac{1}{2}}|\wu(\h)|e^{\beta|\j|}|\j|^{\frac{1}{2}}|\wv(\j)||k|^{r+\frac{1}{2}}|\ww(\k)|e^{\beta |\k|}  \nn \\
& & \quad + \quad 2^{r-1/2}\hspace{-10pt}\sum_{\h+\j-\k=\0}e^{\beta |\h|}|\h|^{\frac{1}{2}}|\wu(\h)|e^{\beta |\j|}|\j|^{r+\frac{1}{2}}|\wv(\j)||\k|^{r+\frac{1}{2}}|\ww(\k)|e^{\beta |\k|},  \nn \\
& & \lesssim 2^r\left( \|A^{\frac{1}{4}}e^{\beta A^{1/2}}\bv\|_{\cW}\|A^{\frac{1}{4}}\bu\|_\beta 
\|A^{\frac{1}{4}}\bw\|_\beta + \|A^{\frac{1}{4}}e^{\beta A^{1/2}}\bu\|_{\cW}\|A^{\frac{1}{4}}\bv\|_\beta \|A^{\frac{1}{4}}\bw\|_\beta \right),\nn \\
& & \lesssim 2^r\Cwien[r]\left(
\|\bv\|_\beta \|A^{\frac{1}{4}}\bu\|_\beta 
\|A^{\frac{1}{4}}\bw\|_\beta + \|\bu\|_\beta \|A^{\frac{1}{4}}\bv\|_\beta \|A^{\frac{1}{4}}\bw\|_\beta \right), \label{final}
\ea
where to obtain \eqref{final} we used \eqref{Sob1} with $s := r - \frac12 > \frac{d}{2}$.
We readily obtain
\be  \label{keyineq}
|\l B(\bu,\bu),A^{r}e^{2\beta A^{1/2}}\bu\r| \lesssim 2^r\Cwien[r]\|\bu\|_\beta \|A^{1/4}\bu\|_\beta^2.
\ee

\end{proof}

 \subsection{Proof of Theorem \ref{thm:euler}.}
 \begin{proof}
 Recall that for each $N \in \N$, $H_N=\Span\{\e_1, \cdots , \e_n\}\subset H_\C$, where $\{\e_i\}_{i=1}^\infty$ is the complete, orthonormal system (in $H_\C$) of eigenvectors of $A$. Denote the orthogonal projection on $H_N$ by $P_N$.
 The Galerkin system corresponding to \eqref{cxeuler} is given by
 \be \label{galerkin}
 \frac{d\bu_N}{d\zeta} + P_N B(\bu_N,\bu_N) = 0,\ \bu_N(0)=P_N\bu_0,\ \bu_N(\zeta) \in H_N.
 \ee
 The Galerkin system is an ODE with a quadratic nonlinerity. Therefore it admits a unique solution  in a neighborhood of the origin in $\C$. We will obtain \textit{a priori} estimates 
 on the Galerkin system in ${\cal R}$ (defined in \eqref{cxregion}) independent of $N$. This will show that the Galerkin system corresponding to \eqref{cxeuler} has a solution for all $\zeta \in {\cal R}$ and
 forms a normal family on the domain ${\cal R}$. We can then pass to the limit through a subsequence by (the Hilbert space-valued version of) Montel's theorem to obtain a solution of \eqref{cxeuler} in  ${\cal R}$. Since we will obtain estimates independent of $N$, henceforth we will denote by $\bu(\cdot)$ a solution to \eqref{galerkin}, i.e., we will drop the subscript $N$.
 
 Fix $\theta \in [0,2\pi)$, and let 
 \[
 \zeta = s e^{i\theta},\  s>0.
 \]
 We assume that the initial data $\bu_0$ satisfies $\|\bu_0\|_{\beta_0} < \infty $ for some $\beta_0>0$. Fix $\delta >0$, to be chosen later and define the time-varying norm
\[
\bn \bu(\zeta)\bn = \|\bu(\zeta)\|_{\beta_0 - \delta s}.
\]
The corresponding (time-varying) inner product will be denoted by $(( , ))$, i.e.,
\bas
\lefteqn{((\bu,\bv))}\\
 & &  =
 \l A^{r/2}e^{(\beta_0-\delta s)  A^{1/2}}\bu,
 A^{r/2}e^{(\beta_0-\delta s) A^{1/2}}\bv\r \\
& & = \l \bu,
A^re^{2(\beta_0-\delta s) A^{1/2}}\bv\r. 
\eas
Taking the inner-product of \eqref{cxeuler} (in $H_{\C}$) with $A^re^{2(\beta_0-\delta s) A^{1/2}}\bu$, then multiplying by $e^{i\theta}$, and finally taking the real part of the resulting equation, we readily obtain
\[
	\frac12 \ds \bn \bu(\zeta)\bn^2 + \delta \, \bn A^{1/4}\bu(\zeta)\bn^2= -Re\left(e^{i\theta} ((B(\bu(\zeta),\bu(\zeta)),\bu(\zeta)))\right)
	\leq | ((B(\bu(\zeta),\bu(\zeta)),\bu(\zeta))) |.
\]
For $s<\frac{\beta_0}{\delta}$, using Proposition \ref{prop:nocancellation},  we obtain
\be  \label{cxineq}
 \frac12 \ds \bn \bu\bn^2 + \delta \, \bn A^{1/4}\bu\bn^2
 \lesssim  2^r\Cwien[r]\bn \bu \bn\bn A^{1/4}\bu\bn^2.
 \ee
 Now choose 
  \[
	  \delta = C2^r\Cwien[r] \|\bu_0\|_{\beta_0}. 
 \]

 From \eqref{cxineq},
we see that $\bn \bu \bn$ is non-increasing  and
\[
	\bn \bu(\zeta)\bn \le \|\bu_0\|_{\beta_0}\ \forall\ \zeta =se^{i\theta},\ 0 < s < \frac{\beta_0}{\delta}.
\]
In particular, this means
\[
	\sup_{\zeta \in {\cal R}}\|A^{r/2}\bu(\zeta)\| \le \|\bu_0\|_{\beta_0}.
\]
As remarked above, the proof is now complete by invoking Montel's theorem.
\end{proof}
 
 \section{Surface Quasi-geostrophic Equations}  \label{sec:SQG}
 We consider the inviscid, two-dimensional (surface) quasi-gesotrophic equation (henceforth SQG) on 
$\Omega=[0,2\pi]^2$, given by 
\begin{gather}    \label{QG}
\begin{alignat}{2}
\partial_t \eta +\bu\cdot \nabla \eta =0&, \quad&& \text{in }\Omega\times \R_+,\\
\bu = [-R_2\eta,R_1\eta]^T&, \quad&& \text{in }\Omega\times \R_+,\\
\eta(0) =\eta_0&, \quad&& \text{in }\Omega.
\end{alignat}
\end{gather}
Here $\bu$ is the velocity field, $\eta$ is the temperature, the operator $\Lambda = (-\Delta)^{1/2}$
(with $\Delta $ denoting the Laplacian) and
the operators $R_i = \partial_i\Lambda^{-1}, i=1,2$,  are the usual Riesz transforms. 

Observe that by the definition of $u$, it is divergence-free. Also, without loss of generality, we will take $u$ and $\eta$ to be mean-free, i.e.,
\[
\int_{\Omega} \bu =\0, \int_{\Omega} \eta =0.
\]

The SQG was introduced in \cite{cmt} and variants of it
arises in geophysics and meteorology (see, for instance \cite{Pedlosky}).
Moreover, the critical SQG is the dimensional analogue of the three dimensional Navier-Stokes equations. Existence and regularity issues for the viscous and inviscid cases were first extensively examined in \cite{resnick}. This equation, particularly the dissipative case with various fractional orders of dissipation,  has received considerable attention recently;  see \cite{caff:vass,  cwu,  kiselev:note, kis:naz:vol}, and the references therein. As in section~\ref{sec:Euler}, our focus here is time analyticity of the inviscid SQG, with values in an appropriate Gevrey class.
 
 As before, for $r > \frac 32 , \beta>0$, we define
 \[
 \|\eta\|_\beta= \|\Lambda^{r} e^{\beta \Lambda}\eta\|\ \mbox{and}\ 
 \|\bu\|_\beta = \|A^{r/2}e^{\beta A^{1/2}}\bu\|.
 \]
  Note that because $\bu$ is the Riesz transform of $\eta$, we have 
 $\|\eta\|_{\beta} \sim \|\bu\|_{\beta}$.
 
  \begin{theorem}
 Let $\eta_0$ be such that $\|\eta_0\|_{\beta_0} < \infty $ for some $\beta_0 >0$.
 The complexified inviscid SQG equation
 \be \label{cxsqg}
 \frac{d\eta}{d \zeta} + B(\bu,\eta)=0, \eta(0)=\eta_0, \ \mbox{where}\ B(\bu,\eta)=\bu\cdot \nabla \eta,
 \ee
 admits a unique solution in the region 
 \be  \label{cxsqgregion}
 {\cal R}= \left\{\zeta=se^{i\theta} : \theta \in 
 [0,2\pi), 0 < s < \frac{C\beta_0}{ 2^{r}\Cwien[r]\|\eta_0\|_{\beta_0}}\right\}.
 \ee
 \end{theorem}
 \begin{proof}
 Proceeding in a similar manner as in Proposition \ref{prop:nocancellation}, we obtain
 \begin{multline}
|\l B(\bu,\eta),\Lambda^{2r}e^{2\beta \Lambda}\eta\r|  \lesssim 2^r\Cwien[r]
\left(\|\eta\|_{\beta}\|A^{1/4}\bu\|_{\beta}\|\Lambda^{1/2}\eta\|_{\beta} +  \|\bu\|_{\beta}\|\Lambda^{1/2}\eta\|^2_{\beta}\right)\\
 \lesssim 2^r\Cwien[r] \|\eta\|_{\beta}\|\Lambda^{1/2}\eta\|^2_{\beta},
 \label{sqg:nocancellation}
\end{multline}
 where the last inequality follows by noting that $u$ is the Riesz transform of $\eta$.
 
 Fix $\theta\in[0,2\pi)$. Let
 \[
 \zeta = s e^{i\theta}, s>0.
 \]
The initial data $\eta_0$ satisfies $\|\eta_0\|_{\beta_0} < \infty $ for some $\beta_0>0$. Fix $\delta >0$, to be specified later, and define the time-varying norm
\[
	\bn \eta \bn = \|\eta(\zeta)\|_{\beta_0-\delta s},
\]
and the corresponding (time-varying) inner product, $(( , ))$, as we did in the proof of Theorem~\ref{thm:euler}.
\comments{
\bas
\lefteqn{((\eta_1,\eta_2))}\\
 & &  =
\l e^{(\beta_0-\delta s \cos \theta)  \Lambda}\eta_1,
e^{(\beta_0-\delta s \cos \theta)  \Lambda}\eta_2\r \\
& & = \l \eta_1,
e^{(2\beta_0-2\delta s \cos \theta)  \Lambda}\eta_2\r. 
\eas
}

Taking the inner-product of \eqref{cxsqg} (in $H_{\C}$) with $\Lambda^{2r}e^{2(\beta_0-\delta s) \Lambda}\eta$, multiplying by $e^{i\theta}$ and taking the real part, we obtain
 \[
	 \frac12 \ds \bn \eta (\zeta)\bn^2 + \delta\, \bn \Lambda^{1/2}\eta(\zeta)\bn^2=Re\left(- e^{i\theta} ((B(\bu(\zeta),\eta(\zeta)),\eta(\zeta))) \right).
\]
 Using \eqref{sqg:nocancellation}, we deduce
 \be  \label{cxetaineq}
 \frac12 \ds \bn \eta \bn^2 + \delta\, \bn \Lambda^{1/2}\eta\bn^2
 \lesssim  2^r\Cwien[r]\bn \eta \bn
\, \bn \Lambda^{1/2}\eta\bn^2.
 \ee
 Now choose 
 \[
 \delta =  C2^r\Cwien[r]\|\eta_0\|_{\beta_0}.
 \]
 From \eqref{cxetaineq},
we see that $\bn \eta \bn$ is non-increasing  and
\[
	\bn \eta(\zeta)\bn \le \|\eta_0\|_{\beta_0}\ \forall\ \zeta =se^{i\theta},\ 0 < s < \frac{\beta_0}{\delta}.
\]
In particular, this means
\[
\sup_{z \in {\cal R}}\|\eta(z)\| \le \|\eta_0\|_{\beta_0}.
\]
This establishes a uniform bound on the Galerkin system and the proof is complete by invoking Montel's theorem as before.
 \end{proof}

\section{Inviscid Boussinesq Equations}
The inviscid  Boussinesq system (without rotation) in the
periodic  domain $\Omega:=[0,2\pi]^d, d=2,3$, for time $t \ge 0$ is given by
\begin{subequations}\label{Bouss}
\begin{alignat}{2}
\label{Bouss_mo}
\partial_t \bu + (\bu\cdot\nabla)\bu +\nabla p&=  \eta\, g \mathbf{e},
\quad&& \text{in }\Omega\times \R_+,\\
\label{Bouss_den}
\partial_t\eta + (\bu\cdot\nabla)\eta &=0,
\quad&& \text{in }\Omega\times \R_+,\\
\label{Bouss_div}
\nabla \cdot \bu &=0,
\quad&& \text{in }\Omega\times\R_+,\\
\label{Bouss_init}
\bu(0)=\bu_0,
\quad\eta(0)&=\eta_0,
\quad&& \text{in }\Omega,
\end{alignat}
\end{subequations}
equipped with periodic boundary conditions in space.  
Here $\mathbf e$ denotes the unit vector in $\R^d$ pointing upward
and $g$ denotes the (scalar) acceleration due to gravity. 
 The unknowns are the fluid velocity field
$\bu$, the fluid pressure
$p$, and the function $\eta$, which may be
interpreted physically as the temperature. 
The Boussinesq system arises in the study of atmospheric, oceanic and astrophysical turbulence, particularly where rotation and stratification play a dominant role 
\cite{Pedlosky, salmon}. 
We will follow the notation for the norms  
as in Section~\ref{sec:Euler} and Section~\ref{sec:SQG}

\begin{theorem}
 Let $(\bu_0,\eta_0)$ be such that $\|(\bu_0,\eta_0)\|_{\beta_0} < \infty $ for some $\beta_0 >0$, where $\|(\bu_0,\eta_0)\|_{\beta_0}^2=\|\bu_0\|_{\beta_0}^2 + \|\eta_0\|_{\beta_0}^2$.
 The complexified inviscid Boussinesq equations \eqref{Bouss}
 admit a unique solution $(\bu(\zeta),\eta(\zeta))$ in the region 
 \ba  
 & & {\cal R}= \left\{\zeta=se^{i\theta}: \theta \in 
 [0,2\pi), \right. 
\left. 0 < s < 
 \min\left\{\frac{C\beta_0}{ 2^r\Cwien[r]\|(\bu_0,\eta_0)\|_{\beta_0}},  \frac{2\ln2}{g}\right\}\ \right\}. \label{boussregion}
 \ea
 \end{theorem}
\begin{proof}
	We proceed as in Section \ref{sec:Euler} and Section \ref{sec:SQG} by taking the inner product of the complexified versions of \eqref{Bouss_mo} and \eqref{Bouss_den} with $A^re^{2(\beta_0-\delta s)A^{1/2}}u$ and $\Lambda^{2r}e^{2(\beta_0-\delta s)\Lambda}\eta$ respectively, then multiplying by $e^{i\theta}$ and taking the real part. Using \eqref{nocancellation} and \eqref{sqg:nocancellation} and adding the results,
for $(\bu(\zeta),\eta(\zeta)), \zeta = s e^{i\theta}$, we obtain
\ba
& & \frac12 \ds (\bn \bu \bn^2+ \bn \eta \bn^2) + 
\delta (\bn \Lambda^{1/2}\bu\bn^2+ \bn \Lambda^{1/2} \eta\bn^2 ) \nn \\
& & \lesssim 2^r\Cwien[r](\bn \bu\bn+\bn \eta \bn)
(\bn \Lambda^{1/2}\bu\bn^2+ \bn \Lambda^{1/2} \eta\bn^2 )+g \bn \eta \bn \bn \bu\bn \nn \\
& & \le 2^r\Cwien[r](\bn \bu\bn+\bn \eta \bn)(\bn \Lambda^{1/2}\bu\bn^2+ \bn \Lambda^{1/2} \eta\bn^2 )+\frac{g}{2}(\bn \bu\bn^2 + \bn \eta \bn^2).
\ea
Thus, as long as  
\be \label{deltacond}
(\bn \bu \bn + \bn \eta \bn) \lesssim \frac{\delta}{2^r\Cwien[r]},
\ee
by the Gronwall inequality, we have 
\be \label{gronwall}
\bn \bu\bn^2 + \bn \eta \bn^2 \le e^{T g}(\|\bu_0\|_{\beta_0}^2 + 
\|\eta_0\|_{\beta_0}^2), 0<s\le T.
\ee
Using the fact that $(a+b) \le \sqrt {2(a^2+b^2)}$, as long as \eqref{deltacond} holds,
from \eqref{gronwall} we have
\be \label{finalbd}
(\bn \bu\bn + \bn \eta \bn) \le e^{\frac{Tg}{2}} \sqrt 2\|(\bu_0,\eta_0)\|_{\beta_0}.
\ee
Now choose 
\bes
\delta = C2^r\Cwien[r]\|(\bu_0,\eta_0)\|_{\beta_0}.
\ees
For all $0<s<T=\min\{\frac{\beta_0}{\delta}, \frac{2\ln 2}{g}\}$, \eqref{deltacond} is satisfied and consequently, \eqref{finalbd} holds.

\end{proof}

\section{Inviscid Magnetohydrodynamic Equations}
The inviscid incompressible magnetohydrodynamic system in the
periodic domain $\Omega:=[0,2\pi]^d, d=2,3$, for time $t \ge 0$ is given by
the following system 
\begin{subequations}\label{MHD_original}
\begin{alignat}{2}
\label{MHD_original_v}
\partial_t \bu + (\bu\cdot\nabla)\bu -\tfrac{1}{S}(\bb\cdot\nabla)\bb +\nabla (\tfrac{1}{\rho_0}p + \tfrac{|\bb|^2}{2S})&= \0,
\quad&& \text{in }\Omega\times \R_+,\\
\partial_t \bb + (\bu\cdot\nabla)\bb - (\bb\cdot\nabla)\bu &= \0,
\quad&& \text{in }\Omega\times \R_+,\\
\label{MHD_original_div}
\nabla \cdot \bu =0,\quad \nabla \cdot \bb &=0,
\quad&& \text{in }\Omega\times\R_+,\\
\label{MHD_original_init}
\bu(0)=\bu_0,
\quad \bb(0)&= \bb_0,
\quad&& \text{in }\Omega,
\end{alignat}
\end{subequations}
equipped with periodic boundary conditions in space. Here, $\bu$ represents the fluid velocity field, $\bb$ the magnetic field and $p$ the fluid pressure. The constant $\rho_0$ is the fluid density, and $S = \rho_0\mu_0$, where $\mu_0$ is the permeability of free space. 

The magnetohydrodynamic equations govern the evolution of an electrically conductive fluid under the influence of a magnetic field, and so are useful in the design of fusion reactors, or the study of solar storms and other natural phenomenon. See \cite{MHD-background} for more on the derivation of \eqref{MHD_original}, and \cite{MHD-examples1,MHD-examples4} for some applications of the magnetohydrodynamic equations (MHD). The existence and uniqueness of solutions to the incompressible MHD has been studied for the viscous case in \cite{MHD-existence,MHD-existence2}, for example, and in \cite{MHD-existence-inviscid1, MHD-existence-inviscid} for the inviscid case (which we consider in this paper). The space analyticity of solutions of \eqref{MHD_original} is discussed in \cite{MHD-analyticity}, whereas in the present work we give criteria for solutions to be holomorphic functions of both the time and space variables.

By rewriting the equations in terms of the Els\"asser variables (which are defined via the transformations $\bv = \bu + \tfrac{1}{\sqrt{S}}\bb$, $\bw = \bu - \tfrac{1}{\sqrt{S}}\bb$), we obtain the equivalent system
\begin{subequations}\label{MHD}
\begin{alignat}{2}
\label{MHD_v}
\partial_t \bv + (\bw\cdot\nabla)\bv +\nabla {\scriptstyle \mathcal{P}}&= \0,
\quad&& \text{in }\Omega\times \R_+,\\
\partial_t \bw + (\bv\cdot\nabla)\bw +\nabla {\scriptstyle \mathcal{P}}&= \0,
\quad&& \text{in }\Omega\times \R_+,\\
\label{MHD_div}
\nabla \cdot \bv =0,\quad \nabla \cdot \bw &=0,
\quad&& \text{in }\Omega\times\R_+,\\
\label{MHD_init}
\bv(0)=\bv_0,
\quad \bw(0)&= \bw_0,
\quad&& \text{in }\Omega,
\end{alignat}
\end{subequations}
where ${\scriptstyle \mathcal{P} }= \tfrac{1}{\rho_0}p + \tfrac{|\bv - \bw|^2}{8}$.

\begin{theorem}
 Let $(\bv_0,\bw_0)$ be such that $\|(\bv_0,\bw_0)\|_{\beta_0} < \infty $ for some $\beta_0 >0$. The complexified inviscid magnetohydrodynamic equations \eqref{MHD} admit a unique solution $(\bv(\zeta),\bw(\zeta))$ in the region 
 \ba  
 & & {\cal R}= \left\{\zeta=se^{i\theta}: \theta \in 
 [0,2\pi), \nn 0 < s < 
 \frac{C\beta_0}{ 2^r\Cwien[r]\|(\bv_0,\bw_0)\|_{\beta_0}}\right\}.\label{MHDregion}
 \ea
 \end{theorem}
\begin{proof}
Proceeding as in the previous sections and using \eqref{nocancellation}, 
for a fixed $\theta\in[0,2\pi)$, for $(\bv(\zeta),\bw(\zeta)), \zeta = s e^{i\theta}$, we obtain
\ba
& & \frac12 \ds \{ \bn \bv \bn^2+ \bn \bw \bn^2\} + 
\delta (\bn \Lambda^{1/2}\bv\bn^2+ \bn \Lambda^{1/2} \bw\bn^2 ) \nn \\
& & \lesssim 2^r\Cwien[r](\bn \bv\bn+\bn \bw \bn)(\bn \Lambda^{1/2}\bv\bn^2+ \bn \Lambda^{1/2} \bw\bn^2 ) \nn \\
& & \lesssim 2^r\Cwien[r]\bn(\bv,\bw)\bn(\bn \Lambda^{1/2}\bv\bn^2+ \bn \Lambda^{1/2} \bw\bn^2 ).\label{mhdineq}
\ea

Now choose 
 \[
	 \delta = C2^r\Cwien[r]\|(\bv_0,\bw_0)\|_{\beta_0}.
 \]
 From \eqref{mhdineq},
we see that $\bn (\bv,\bw) \bn^2$ is non-increasing  and
\[
	\bn (\bv(\zeta),\bw(\zeta))\bn \le \|(\bv_0,\bw_0)\|_{\beta_0}\ \forall\ \zeta =se^{i\theta},\ 0 < s < \frac{\beta_0}{\delta}.
\]
In particular, this means
\[
\sup_{z \in {\cal R}}\|(\bv(z),\bw(z))\| \le \|(\bv_0,\bw_0)\|_{\beta_0}.
\]
This finishes the proof.

\end{proof}

\section{Analytic Nonlinearity}
In this section, we consider the more general case of an analytic nonlinearity on our basic spatial domain $\Omega:=[0,2\pi]^d$. Again, we consider an equation without viscous effects (see \cite{feti} for the dissipative version). For simplicity of exposition, we only consider the case of a scalar equation here. A vector-valued version, i.e. the case of a system, can be handled in precisely the same way, although notationally it becomes more cumbersome.
Let 
\[
F(z)= \displaystyle\sum_{n=1}^\infty a_n z^n
\]
be a real analytic function in a neighborhood of the origin. 
\comments{Here $\z=(z_1,\cdots,z_n) \in \R^n$ and we employ the multi-index convention $\z^\k = z_1^{k_1}\cdots z_n^{k_n}$ for 
$\k=(k_1,\cdots,k_n)$.}
  The ``majorizing function" for $F$ is defined to be
\be \label{maj}
F_M(s) = \sum_{n =1}^\infty|a_n|s^n, \ \ s<\infty.
\ee
The functions $F$ and $F_M$ are clearly analytic in the open balls (in $\R^d$ and $\R$ respectively) with  center zero and radius
\be \label{rad}
R_M= \sup \left\{s:F_M(s) < \infty\right\}.
\ee
We assume that $R_M>0$. The derivative of the function $F_M$, denoted by $F_M'$, is also analytic in the ball of radius $R_M$. Therefore, for any fixed $r>0$,  the function $\tF$, defined by
\be  \label{tf}
\tF(s)=\sum_{n=1}^\infty |a_n| n^{r+\frac32}(\Cwien[r])^{n-1} s^{n-1}, s \in \R,
\ee
is  analytic in the ball of radius $R_M/\Cwien[r]$. Moreover, 
\be  \label{tfinc}
\tF(s) \ge 0\  \mbox{for }\ s \ge 0\ \mbox{and}\ \tF(s_1) < \tF (s_2)\ \mbox{for}\ 0\le s_1 < s_2.
\ee

\comments
{We now consider the nonlinearity $G$  of the type
\be \label{nonlin}
G(u) = F \left(T_1 u, \cdots,T_n u \right),
\ee
for a scalar function $u:\Omega \ra \R$, 
where $T_i, 1 \le i \le n$ are given by
\[
\wh{T_iu}(\k)=m_{T_i}(\k)\wh{u}(\k), |m_{T_i}(\k)| \le C|\k|^{\alpha_i}, 1 \le i \le n.
\]
For instance, if $T_j = \partial_{x_i}$, then $m_{T_j}(\k)= \imath k_i, \k=(k_1,\cdots,k_d)$.
}

We will  consider an inviscid equation of the form
\be \label{eqn:analytic}
\partial_t u = TF(u), u(0)=u_0,
\ee
where $T$ is given by
\[
\wh{Tu}(\k)=m_{T}(\k)\wh{u}(\k), |m_{T}(\k)| \le C|\k|, \k \in \tZ^d.
\]
We will assume that \eqref{eqn:analytic} preserves the mean free condition under 
evolution. Here, the phase space $H=\dot{L^2}(\Omega)$ and 
$A=(-\Delta)|_H$. As before, we fix $r > \frac{d+1}{2}$ and define
\[
\|u\|_{\beta} = \|A^{\frac r2}e^{\beta A^{1/2}}u\|.
\]
The following proposition is elementary.
\begin{prop}  \label{prop:elementary}
For $x_1, \cdots , x_n \in \R_+$ and any $r>0$, we have
\bes  
(x_1+\cdots +x_n)^r \le n^r(x_1^r+\cdots +x_n^r).
\ees
\end{prop}
\begin{proof}
Without loss of generality, assume $x_1=\max\{x_1, \cdots , x_n\} >0$.
Let $\xi_i = \frac{x_i}{x_1}$ and note that $0 \le \xi_i \le 1$. Then,
\bes
(\sum_{i=1}^n x_i)^r 
 = x_1^r(\sum_{i=1}^n \xi_i)^r \le x_1^r(\sum_{i=1}^n 1)^r
 =n^rx_1^r \le n^{r}\sum_{i=1}^n x_i^r.
\ees

\end{proof} 
We will need the following estimate of the nonlinear term to proceed.
\begin{prop}  \label{prop:analytic}
Let $u \in H_\C$ with $\|A^{1/4}u\|_{\beta} < \infty $. 
Then
\be  \label{ineq:nonlin}
|\l TF(u), A^re^{2\beta A^{1/2}}u\r| \lesssim \tF(\|u\|_\beta)\|A^{1/4}u\|_\beta^2.
\ee
\end{prop}
\begin{proof}
Observe that for $\h_1+\cdots +\h_n+\k=\0,\ \h_i, \k \in \tZ^d$, 
by triangle inequality and Proposition \ref{prop:elementary}, we have
\be  \label{elemineq}
|\k|^r \le n^r (|\h_1|^r + \cdots +|\h_n|^r).
\ee
Denote 
\[
I \subset \tZ^{d+1}, I=\{(\h_1,\cdots ,\h_n,\k):\h_1+\cdots +\h_n+\k =\0, \h_i, \k \in \tZ^d\}.
\]
Thus,
\ba
\lefteqn{|\l T u^n,A^re^{2\beta A^{1/2}}u\r|} \nn \\
& & \lesssim \sum_{I}|u(\h_1)|\cdots |u(\h_n)||u(\k)||\k|^{2r+1}e^{2\beta |\k|} \nn \\
& & \lesssim n^r \left(\sum_{I}|\h_1|^{r}e^{\beta |\h_1|}|u(\h_1)|\cdots 
e^{\beta |\h_n|}|u(\h_n)||u(\k)||\k|^{r+1}e^{\beta |\k|}\right.\nn \\
& & \left. \qquad \qquad \qquad  + \cdots + 
\sum_{I}e^{\beta |\h_1|}|u(\h_1)|\cdots |\h_n|^{r}e^{\beta |\h_n|}|u(\h_n)||u(\k)||\k|^{r+1}\e^{\beta |\k|}\right), \label{intmed1}
\ea
where to obtain \eqref{intmed1}, we used \eqref{elemineq} as well as the triangle inequality
$|\k| \le \sum_i |\h_i|$.
Because 
$\min\{|\h_1|, \cdots , |\h_n|, |\k|\} \ge 1$, we have
\[
|\k| \le \sum_i |\h_i| \le n |\h_1|\cdots |\h_n|,\ \mbox{which implies}\ 
|\k|^{\frac{1}{2}} \lesssim n^{1/2}|\h_1|^{\frac{1}{2}}  \cdots |\h_n|^{\frac12}.
\]
Consequently, from \eqref{intmed1}, we conclude

\ba
\lefteqn{|\l T u^n,A^re^{2\beta A^{1/2}}u\r|} \nn \\
& &\lesssim
n^{r+\frac12} \left(\sum_{I}|\h_1|^{r+\frac12}e^{\beta |\h_1|}|u(\h_1)|\cdots e^{\beta |\h_n|}|\h_n|^{\frac12}|u(\h_n)||u(\k)||\k|^{r+\frac12}e^{\beta |\k|}\right.\nn \\
& & \left. \qquad \qquad   + \cdots +
 \sum_{I}e^{\beta |\h_1|}|\h_1|^{\frac12}|u(\h_1)|\cdots
e^{\beta |\h_n|} |\h_n|^{r+\frac12}|u(\h_n)||u(\k)||\k|^{r+\frac12}e^{\beta |\k|}\right) \nn \\
& & \lesssim n^{r+\frac32}(\Cwien[r])^{n-1} \|A^{\frac14}u\|_\beta^2\|u\|_\beta^{n-1},  \label{intmed2}
\ea
where the last inequality follows exactly as in the proof of  \eqref{final}. This immediately  yields \eqref{ineq:nonlin}.

\end{proof}

 \begin{theorem}  \label{thm:analnonlin}
 Let $r >\frac{d+1}{2}$ and $\beta_0>0$ be fixed and $u_0$ be such that 
 $\|u_0\|_{\beta_0} < \infty $. Then,
 the complexified equation \eqref{eqn:analytic}
 admits a unique solution in the region 
 \bes  
 {\cal R}= \left\{z=se^{i\theta}: \theta \in 
 [0,2\pi), 0 < s < \frac{C\beta_0}{\tF(\|u_0\|_{\beta_0})}\right\}.
 \ees
 \end{theorem}
 
 \begin{proof}
 Fix $\delta >0$, to be chosen later and, as before,  define the time-varying norm
\[
\bn u(\zeta)\bn = \|u(\zeta)\|_{\beta_0 - \delta s}.
\]
Recall that the corresponding (time-varying) inner product is denoted by $(( , ))$, i.e.,
\bas
\lefteqn{((u,v))}\\
 & &  =
 \l A^{r/2}e^{(\beta_0-\delta s)  A^{1/2}}u,
A^{r/2}e^{(\beta_0-\delta s) A^{1/2}}v\r \\
& & = \l u,
A^{r}e^{2(\beta_0-\delta s) A^{1/2}}v\r. 
\eas
Multiplying \eqref{eqn:analytic} by $e^{i\theta}$, taking the real part and then the inner-product with $A^re^{2(\beta_0-\delta s) A^{1/2}}u$, we readily obtain
\[
\frac12 \ds \bn u(\zeta)\bn^2 + \delta\, \bn A^{1/4}u(\zeta)\bn^2=
- ((\ Re(e^{i\theta}F(u(\zeta)) ) , u(\zeta)\ )), \ \zeta = s e^{i\theta}.
\]
Using Proposition \ref{prop:analytic}, we obtain
\be  \label{cxineq:analytic}
 \frac12 \ds \bn u\bn^2 + \delta\, \bn A^{1/4}u\bn^2
 \lesssim  \tF(\bn u \bn)\bn A^{1/4}u\bn^2.
 \ee
 
 Now choose 
 \[
 \delta = C \tF(\|\bu_0\|_{\beta_0}).
 \]
 From \eqref{cxineq:analytic}, and the fact that $\tF (\cdot)$ is strictly increasing 
 \eqref{tfinc}, 
we see that $\bn u \bn$ is non-increasing  and
\[
	\bn u(\zeta)\bn \le \|u_0\|_{\beta_0}\ \forall\ \zeta =se^{i\theta},\ 0 < s < \frac{\beta_0}{\delta}.
\]
In particular, this means
\[
\sup_{z \in {\cal R}}\|u(z)\| \le \|u_0\|_{\beta_0}.
\]
As before, the proof is now complete by invoking Montel's theorem.

\end{proof}

\begin{rem}
{\em
One can extend the method of this section to handle a nonlinearity of the form 
\[
F(u)=T_0G(T_1u,\cdots , T_nu), 
\]
where $G$ is an analytic function of $n$-variables and $T_i$ are Fourier multipliers with
symbol $m_i$ satisfying
\[
|m_i(\k)| \lesssim |\k|^{\alpha_i}\ \forall\ \k \in \tZ^d, 0\le i \le n, \sum_{i=0}^n \alpha_i \le 1.
\]
Using the exact same technique, one can in fact also consider the case of systems, in which case Theorem \ref{thm:euler} becomes a special case.
}

\end{rem}

\section*{Acknowledgement}
This research was partially supported by the NSF grant DMS-1517027 and the CNMS start-up fund of the University of Maryland, Baltimore County.

\end{document}